\documentclass[14pt]{amsart}
\usepackage{amsaddr}

\usepackage{amsmath, amsthm, amssymb}
\usepackage{amscd}
\usepackage{mathtools}

\usepackage{tikz}
\usetikzlibrary{positioning}

\tikzset{>=stealth}

\usepackage{wrapfig}
\usepackage{epstopdf}
\usepackage{epsfig}

\usepackage{chngcntr}
\usepackage{caption}
\usepackage{subcaption}

\usepackage{graphicx}

\usepackage{amsmath,color, url}
\usepackage{amsfonts, hyperref}
\usepackage{tikz}
\usetikzlibrary{positioning}
\usetikzlibrary{calc}
\tikzset{>=stealth}

\DeclareMathOperator{\mult}{mult}

\DeclareMathOperator{\Div}{Div}
\DeclareMathOperator{\PDiv}{PDiv}

\DeclareMathOperator{\divv}{div}

\newcommand{\ph}{\varphi}

\newcommand{\hC}{\widehat{\mathbb{C}}}

\newcommand{\Spin}{\mathrm{Spin}}
\newcommand{\Jac}{\mathrm{Jac}}%for Jacobian

\makeatletter
\newcommand\xxleftrightarrow[2][]{%
  \ext@arrow 9999{\longleftrightarrowfill@}{#1}{#2}}
\newcommand\longleftrightarrowfill@{%
  \arrowfill@\leftarrow\relbar\rightarrow}
\makeatother

\newcommand{\C}{\mathbb{C}}
\newcommand{\Z}{\mathbb{Z}}

\newtheorem{theorem}{Theorem}%[section]
\newtheorem{corollary}[theorem]{Corollary}
\newtheorem{remark}{Remark}
\newtheorem{example}{Example}

\theoremstyle{definition}
\newtheorem{definition}{Definition}

\graphicspath{{./figures/}}

\begin{document}
\author{Yahya Almalki$^{1}$ \and Craig A. Nolder$^2$}
%\author{Yahya Almalki}
%\address[A. One]{Florida State University   \\ King Khalid University}
%\email[A. One]{aone@@aoneinst.edu}
%\thanks{Thanks for Author One.}
%\author{Author Two}
%\curraddr[A. Two]{Florida State University}
%\email[A.~Two]{atwo@@atwoinst.edu}
\address{Florida State University$^{1,2}$   \ \ \     King Khalid University$^1$}
\email{ yaa09d@my.fsu.edu $\bullet$  nolder@math.fsu.edu}
\title{Spin Structures and Branch Divisors on $p$-gonal  Riemann surfaces  }
\maketitle

\tableofcontents

\begin{abstract}
%% Text of abstract

We study spin structures on Riemann and Klein surfaces in terms of divisors.  In particular, we take a closer look at spin structures on hyperelliptic  and $p$-gonal surfaces defined by divisors supported on their branch points.  Moreover, we study invariant spin divisors under automorphisms and anti-holomorphic involutions of Riemann surfaces.
\end{abstract}
\section{Introduction}
Spin structures are roots of the canonical line bundle on a Riemann surface.  They can be also viewed as divisors and lifts of  covering groups, see~\cite{NP-HSK}. They appear in many areas in mathematics such as Riemannian and algebraic geometry.  Spin structures are important in physics because of their widespread applications.  In this paper, we look at structures  from the viewpoint of complex analysis; in particular, as divisors on Riemann surfaces.
 We start by proving a theorem that gives a presentation of the $2$-torsion, $J_2$, and $p$-torsion, $J_p^\ast$ supported on the branch set, subgroups of the Jacobian of hyperelliptic and their natural generalization as $p$-gonal surfaces.  It turns out this presentation is very useful when studying invariant $m$-spin divisors supported on branch points under automorphisms of hyperelliptic surfaces.  We review a formula obtained by Mumford in \cite{mumford} that gives a presentation of $2$-spin divisors on hyperelliptic surfaces.  Then we prove a similar formula for $m$-spin divisors, for an even $m$.

S. Natanzon extended the definition of spin structures on Riemann surfaces to Klein surfaces.  He studied these structures using Arf functions and liftings of covering groups.  In this work, we rewrite Natanzon's definition using divisors on complex doubles, which are orientable surfaces, of Klein surfaces.  In particular, we study hyperelliptic surfaces admitting anti-holomorphic involutions and count their spin structures.  We generalize spin divisors on Riemann surfaces to the case of Klein surfaces.  The main tool is using complex doubles of Klein surfaces.  Moreover,  we take a closer look at Klein surfaces whose complex double are hyperelliptic surfaces.  In particular, we study invariant spin divisors under anti-holomorphic involutions of Riemann surfaces.

\section{Preliminaries}

\begin{definition} Let $X$ be a compact Riemann surface of genus $g$ and let $K_X$ denote the canonical line bundle on $X$.
  A holomorphic line bundle satisfying $L^{\otimes 2}=K_X$ is called a holomorphic spin bundle or a spin structure.
\end{definition}
It is known that spin bundles exist on a compact Riemann surface of genus $g$ and there are $2^{2g}$ of them, \cite{Bobenko}.  In \cite{atiyah}, Atiyah showed that there is a correspondence between spin bundles and theta characteristics on a compact Riemann surface.
Natanzon generalized the above definition and introduced the concept of $m$-spin bundles, see~\cite{NP-MSH,NP-CSK,NP-HSK}.
\begin{definition}
Let $X$ be a compact Riemann surface of genus $g$ and let $K_X$ denote the canonical line bundle on $X$.
    A holomorphic line bundle satisfying $L^{\otimes m}=K$ is called a holomorphic $m$-spin bundle or an $m$-spin structure where $m$ is a positive integer such that $m|2g-2$.
\end{definition}
According to \cite{Bobenko}, if $m|2g-2$, there are $m^{2g}$ $m$th-roots of the canonical line bundle.  Therefore, there are $m^{2g}$ $m$-spin structures.  Since our ultimate goal is studying spin structures using divisors,  we need to interpret the above definitions in terms of divisors.
 Using the correspondence between classes of linearly equivalent divisors and isomorphic line bundles, we can rewrite the above definitions in terms of divisors.

 \begin{definition}
 Let $X$ be a compact Riemann surface of genus $g$ and let $K$ denote the canonical divisor on $X$.
     A divisor $D$ is an $m$-spin structure or an $m$-spin divisor if $m.D=K$, where $m$ is a positive integer such that $m|2g-2$.
\end{definition}

Let $\Div(X)$ and  $\PDiv(X)$ denote the group of divisors and the group of principal divisors; respectively.  The group $\Jac(X)=\Div_0(X)/\PDiv(X)$ is called the Jacobian of $X$. It is well-known that the Jacobian group of a genus $g$ compact Riemann surface  can be given a complex structure, namely it is a complex torus of dimension $g$. For example the Jacobian of a torus is a torus itself, $J(T)=T$. The Jacobian also is an abelian variety.  A detailed discussion of these facts can be found for example in the book of Miranda~\cite{Miranda}. For our purposes, we only need the group structure of the Jacobian.  Denote by $J_m$ the additive subgroup of $m$-torsion degree zero divisors up to linear equivalence, $J_m(X)=\{D\in \Jac(X)| m.D=0\}$.  Clearly $J_m(X)$ is a subgroup of $\Jac(X)$, see~\cite{donu}.  For $m|2g-2, m\ge 2$, let $\mathrm{Spin}_m(X)$ denote the set of $m$-spin divisors on a compact Riemann surface $X$.  For the rest of this work, we always assume $m|2g-2, m\ge 2$ when talking about $m$-spin structures.
\begin{theorem}\label{spin-Jm}
  There is a one-to-one correspondence  between the set of $m$-spin divisors on a compact Riemann surface $X$, $\mathrm{Spin}_m(X)$,  and the group of $m$-torsion elements of the Jacobian, $J_m(X)$.
  \end{theorem}
 The proof is very straightforward.
    Let $D_0$ be an $m$-spin divisor. Define a map \begin{eqnarray*}
                                                 % \nonumber to remove numbering (before each equation)
                                                   \mu_{D_0}: \mathrm{Spin}_m(X) &\longrightarrow & J_m \\
                                                    D&\longmapsto &  D-D_0 .
                                                 \end{eqnarray*}
This map is well defined: given $D\in \mathrm{Spin}_m(X)$, then $m(D-D_0)=K-K=0$ i.e.  $D-D_0\in J_m$.
     Assume $\mu_{D_0}(D_1)=\mu_{D_0}(D_2)$ i.e. $D_1-D_0=D_2-D_0$, then $D_1=D_2$. Hence, the map $\mu_{D_0}$ is injective.
    Moreover, given $E\in J_m$, then $m(D_0+E)=mD_0+mE=K$, hence $D_0+E \in \mathrm{Spin}_m(X)$. Since $E=\mu_{D_0}(D_0+E)$, the surjectivity follows.

\section{Spin Divisors on the Riemann Sphere $\hC$ and the Torus}

Let $\omega= dz$ be a one-form on the Riemann sphere, $\hC= \mathbb{C}\cup\{\infty\}$.  Let $t=1/z$ then $dz=-1/t^2 dt$.  Hence $\omega$ has a double pole at $\infty$.  Therefore, $\divv(\omega)=-2.\infty$ and the canonical divisor on $\hC$ is $K=-2.\infty$.    On the Riemann sphere, it is not hard to see that any divisor of degree $0$ can be written as a divisor of a meromorphic function, \cite{Miranda}.  Therefore, the Jacobian of the Riemann sphere is trivial $\Jac(\hC)=\Div_0/\PDiv=\{0\}$ and its $m$-torsion subgroups $J_m$ are all trivial.  Since the only positive integer satisfying $m|2g-2, m\ge 2$ is $m=2$ and by the correspondence in Theorem~\ref{spin-Jm},  there is only one $2$-spin divisor.   In particular, the divisor $D=-1.\infty$ is a $2$-spin divisor on $\hC$, $2D=K$.

Let $w_1, w_2$ be two fixed complex numbers such that $w_1/w_2$ is not real. Let $\Lambda=\Lambda(w_1, w_2)=\{n_1w_1+n_2w_2| n_1, n_2 \in \Z \}$ be a lattice over the complex plane. The lattice $\Lambda$ is a subgroup of $\C$ and the quotient $\C/\Lambda$ is a complex torus and it has a group structure, \cite{Miranda}.
Note that for an element $x$ in the torus, viewed as a group, satisfies $mx=0$ in $T$, $m\ge 2$ if and only if $mx\in \Lambda$ i.e. $mx$ is a linear combination of $w_1, w_2$. Therefore,  the $m$-torsion subgroup, $\mathrm{Tor}_m$, of the torus is $\mathrm{Tor}_m=\frac{1}{m}\Lambda=\{a\frac{w_1}{m}+b\frac{w_2}{m}| a, b \in \Z, 0\le a,b\le m-1\}$.

The 1-form $\omega=dz$ is a nowhere zero holomorphic form and its divisor is $\divv(\omega)=0$.  Therefore, the canonical divisor on the torus is $K=0$. Moreover,  it is well known that the Jacobian of the torus is itself, $\Jac(T)=T$, \cite{Miranda}. Hence, the subgroups $\mathrm{Tor}_m\subset T$ and $J_m\subset \Jac(T)$ are isomorphic and can be identified. Note that elements of $\mathrm{Tor}_m$ are points on the torus, but elements of $J_m$ are divisors of degree $0$.  Hence, we identify the point $a\frac{w_1}{m}+b\frac{w_2}{m}$ with the divisor $a.\frac{w_1}{m}+b.\frac{w_2}{m}-(a+b).0$.  By the correspondence in Theorem~\ref{spin-Jm} between $\mathrm{Spin}_m$ and $J_m$, one can write down explicitly all the $m^2$ $m$-spin divisors on the torus.  For example, $D_1=0, D_2=1.w_1/2-1.0, D_3=1.w_2/2-1.0, D_4=1.((w_1+w_2)/2)-1.0$ are the four $2$-spin divisors on the torus.

\section{Hyperelliptic and $p$-Gonal Riemann Surfaces }\label{hyp-p}
 Let $X$ be the compact Riemann surface of genus $g$ defined by $y^2=\prod_{i=1}^{2g+2}(x-e_i)$ or $y^2=\prod_{i=1}^{2g+1}(x-e_i)$ where $x$ and $y$ are complex variables.  A surfaces of this form is called a hyperelliptic Riemann surfaces and it is a degree two branched covering of the Riemann sphere $\pi:X\longrightarrow \hC, (x,y)\longmapsto x$  with $2g+2$ ramification points $A=\{p_1,\ldots, p_{2g+2}\}$ (if the degree of the polynomial is $2g+1$, then $p_{2g+2}=(\infty,0)$) such that $\pi(p_i)=e_i$, see~\cite{Bobenko} and \cite{Miranda}.  Since hyperelliptic and $p$-gonal surfaces can be viewed as covers of the Riemann sphere, we will  compute canonical divisors on these surfaces by pulling-back the canonical divisor on the sphere.

\begin{theorem}\label{unique-mero}
Every meromorphic function $f$ on a hyperelliptic Riemann surface $X$ defined by $y^2=\prod_{i=1}^{2g+2}(x-e_i)$ or $y^2=\prod_{i=1}^{2g+1}(x-e_i)$ can be written uniquely as
$f = r(x) + ys(x)$,
where $r(x)$ and $s(x)$ are rational functions of $x$.
\end{theorem}
Consult~\cite{Miranda} for a proof of this theorem.  The meromorphic function $\pi$ has either a pole of order two if this pole is also a ramification point or it has two simple poles.
In the first case we have  $\pi(p_k)=\infty$, for some $ p_k\in A$. Then the meromorphic function $f=\pi-\pi(p_i), i\neq k$ has an order two pole at $p_k$ and an order two zero at $p_i$.  Therefore $(f)=2.p_i-2.p_k$, we have $2.p_i\sim 2.p_k, \forall p_i$.  On the other hand, if $\forall p_i,\pi(p_i)\neq \infty$, $\pi$ has two simple poles $a_1,a_2$, then $ div(\pi-\pi(p_i))=2.p_i-a_1-a_2$.
Hence, $2.p_i\sim a_1+a_2\forall p_i$.

Let $D=2.p_i$, it follows that, in either of the two above cases, $D\sim 2.p_j, \forall p_j\in A$.  Moreover, we have $2.p_1+\cdots+2. p_{2g+2}\sim (2g+2)(2.p_i)\sim (2g+2)D$.
Notice also the divisor of the meromorphic function $h(x,y)=y/(x-e_j)^{g+1}$ is $(h)=p_1+\cdots+p_{2g+2}-(2g+2).p_j$, hence $p_1+\cdots+p_{2g+2}\sim (2g+2).p_j$.  Let $K_\infty=-2.\infty$ be the canonical divisor on $\hC$.  Notice that the pull-back of the canonical divisor on the sphere is  $\pi^*(K_\infty)=-2.D$ and the ramification divisor is $R_\pi=p_1+\cdots+ p_{2g+2}$.  Therefore, we have an expression of the canonical divisor on the above hyperelliptic surface,   $K_X = \pi^*(K_\infty)+R_\pi\sim -2D+p_1+\cdots+ p_{2g+2}\sim -2D+(g+1)D=(g-1)D $.

Another class of Riemann surfaces that admit an algebraic model are surfaces defined by equations of the form $y^p=(x-e_1)\cdots(x-e_r)$, where $x, y$ are complex variables, $r=np, n\in \mathbb{N}$ , $p$ is a prime, and $g=(p-1)(r-2)/2$. A surface $X$ of this type is known as  a $p$-gonal surface and it is a degree $p$ branched covering of the sphere  $\pi:X\longrightarrow \hC, (x,y)\longmapsto x$ with $r$ branch points $e_1, \ldots, e_r$ and $r$ ramification points $a_i=(e_i,0)\in X$, see~\cite{Miranda}.  The surface $X$ has a $p$-involution $j:X\longrightarrow X, j(x,y)=(x, e^{2\pi i /p}y)$ which fixes the ramification points.
\begin{remark}
  The roots of a polynomial defining a $p$-gonal surface need not be simple in general, but for our purposes, we will assume they are.
\end{remark}
 \begin{theorem}
Every meromorphic function $f$ on a $p$-gonal Riemann surface $X$ defined by $y^p=(x-e_1)\cdots(x-e_r)$ can be written uniquely as
$\displaystyle f =\sum_{i=0}^{p-1}s_i(x)y^i$,
where  $s_i(x)$ are rational functions of $x$.
\end{theorem}
This theorem is similar to Theorem~\ref{unique-mero}. The reader may consult~\cite{Miranda} for proofs and detailed discussion of function fields on $p$-gonal surfaces.

 Similar to what we did in the case of hyperelliptic surfaces, the meromorphic function $x-e_i:X\longrightarrow \hC$ has a zero at $a_i=(e_i,0)$ of order $p$  and $p$ simple zeros $z_1,\ldots ,z_p$ which are permuted under the $p$-involution, see~\cite{gonzalez}.  Hence, the divisor of $x-e_i$ is $(x-e_i)=p.a_i-(z_1+\cdots+z_p)$.  Furthermore,  the meromorphic function $y:X\longrightarrow \hC$ has $r=np$ simple zeros at the ramification points and $p$ poles of order $n$.  Therefore, the divisor of the function $y:X\longrightarrow \hC$ is $(y)=a_1+\cdots a_r-n(z_1+\cdots+z_p)$, for details see~\cite{Miranda} and \cite{gonzalez}. Combining both divisors, we have the relation $a_1+\cdots a_r=np.a_i$. Furhthermore,  simple calculations show that the ramification divisor is $R_\pi=\sum_{x\in X}((\mult_x \pi) -1).x=(p-1)(a_1+\cdots+ a_r)$ and the pullback of the canonical divisor on  the sphere $\hC$ is $\pi^\ast(-2.\infty    )=-2pa_i$.   Using the formula of  the canonical divisor $K_X=R_\pi+ \pi^\ast(-2.\infty)$, we see that   $K_X=-2p.a_i+(p-1)(a_1+\cdots+ a_r)=p(np-(n+2))a_i$.  Using the genus formula, $g=(p-1)(r-2)/2$, we get $K_X=(2g-2)a_i$.  Let $D=p.a_i$, then $K_X=(pn-(n+2))D$.

\begin{remark}
  We say a divisor is supported on the branch set of a surface, if this divisor can be written explicitly in terms of the ramification points.  When studying hyperelliptic and $p$-gonal surfaces, we will not make a distinction between the branch points on the sphere $e_i$ and  the ramification points $a_i=(e_i,0)$ on the surface.  This seems to be the standard terminology in literature, e.g.\cite{gonzalez} and \cite{invariant}.   We will denote the branch set of a surface $X$ by $B(X)$.
\end{remark}
\begin{theorem}
  Let $X$ be a $p$-gonal Riemann surface of genus $g$.  Assume $m|2g-2$, then there exist at least one $m$-spin structure supported on the branch set.
\end{theorem}
\begin{proof}
  Let $\theta=\frac{(2g-2)}{m}. a_0$, where $a_0$ is a branch point.  Clearly, $m.\theta=(2g-2).a_0=K_X$, hence $\theta$ is an  $m$-spin structure.
\end{proof}

\section{Jacobian Groups of $p$-Gonal Surfaces and their Subgroups}\label{number-of-spin}
Let $J_m=\{ E\in \Jac(X)| m.E=0 \}$. This group is known as the subgroup of $m$-torsion points. It is known that $J_m\cong \mathbb{Z}_m^{2g}$ where $g$ is the genus of a Riemann surface $X$, see \cite{donu}. Below we prove some facts about $m$-torsion subgroups that we will use in later sections.
\begin{theorem}On a Riemann surface $X$, the following holds

\begin{enumerate}
               \item $J_{m_1}\bigcap J_{m_2}=J_{\mathrm{gcd}(m_1, m_2)}$.
               \item  $J_{m_1}+ J_{m_2}\subseteq J_{\mathrm{lcm}(m_1, m_2)}$ where $J_{m_1}+J_{m_2}=\{E_1+E_2|E_1, E_2\in \mathrm{Jac}(X), m_1E_1=0, m_2E_2=0 \}$.
\end{enumerate}

\end{theorem}
\begin{proof}
  To prove the first part let $d=\mathrm{gcd}(m_1, m_2)$, then for some $\widehat{m_1}, \widehat{m_2}\in \mathbb{Z}$, we can write $m_1=\widehat{m_1}.d$ and $m_2=\widehat{m_2}.d$. If $E\in J_{\mathrm{gcd}(m_1, m_2)}$ i.e. $d.E=0$, then $m_1.E=0, m_2.E=0$. Hence, $E\in J_{m_1}\bigcap J_{m_2} $.  On the other hand, using Euclid's algorithm, $d=\mathrm{gcd}(m_1, m_2)$ can be written as  $d=s_1m_1+s_2m_2$ where $s_1, s_2\in \mathbb{Z}$.  If  $E\in J_{m_1}\bigcap J_{m_2} $ i.e. $m_1.E=0, m_2.E=0$, then $d.E=0$. Therefore,  $E\in J_{\mathrm{gcd}(m_1, m_2)}$. %Since $\mathrm{gcd}(m_1, m_2)=\frac{m_1m_2}{\mathrm{lcm}(m_1, m_2)}$, it follows that \mathrm{gcd}(m_1, m_2)E=\frac{m_1.m_2}{\mathrm{lcm}(m_1, m_2)}$

For the second part, let $E \in J_{m_1}+J_{m_2} $ i.e. $E$ has the form  $E=E_1+E_2, m_1.E_1=0, m_2.E_2=0$. Since $\mathrm{lcm}(m_1, m_2)=\frac{m_1.m_2}{\mathrm{gcd}(m_1, m_2)}$, it follows that $\mathrm{lcm}(m_1, m_2)E_1=0,\mathrm{lcm}(m_1, m_2) E_2=0 $. Hence,  $\mathrm{lcm}(m_1, m_2).(E_1+E_2)=0$ and $E~=~E_1+E_2 ~\in~ J_{\mathrm{lcm}(m_1, m_2)}$.

\end{proof}

\begin{theorem}\label{J*-hyp}
   On a hyperelliptic Riemann surface $X$, for an even integer $m=2h$, elements of $J_m$ supported on the branch set, denoted by $J_m^\ast$, are in $J_2$.  Hence, we have $J_m^\ast=J_2$.
\end{theorem}
\begin{proof}

Suppose $E=c_1a_1+\cdots + c_la_l \in J^\ast_m, a_i\in B(X)$ i.e. $m.E=0$. Since $2.a_i=2.a_j$, we can simplify $m.E=2h.E= (c_1 (2.a_1)+ \cdots + c_l(2.a_l))= h(\sum^l c_i)(2.a_1)=0$.  Therefore, $\sum^l c_i=0$.  Hence, we have $2.E=(c_1 (2.a_1)+ \cdots + c_l(2.a_l))= (\sum^l c_i)(2.a_1)=0$, therefore, $E\in J_2$.
\end{proof}

\begin{theorem}\label{J*-p}
On a $p$-gonal Riemann surface,   for a non-zero multiple of $p$, $m=h.p$, elements of $J_m$ supported on branch set, denoted by $J_m^\ast$, are in $J_p^\ast$.  Hence, we have $J_m^\ast=J_p^\ast$.
\end{theorem}
\begin{proof}

Suppose $E=c_1a_1+\cdots + c_la_l \in J_m, a_i\in B(X)$ i.e. $m.E=0$.  Since $p.a_i=p.a_j$,  we have $m.E=h. (c_1 (p.a_1)+ \cdots + c_l(p.a_l))= h(\sum^l c_i)(p.a_1)$.  Therefore, $\sum^l c_i=0$.  Hence, we have  $p.E=(c_1 (p.a_1)+ \cdots + c_l(p.a_l))= (\sum^l c_i)(p.a_1)=0$, therefore, $E\in J_p^\ast$.
\end{proof}

\begin{theorem}
  Let $\mathrm{Jac}^\ast$ denote elements of the Jacobian of a $p$-gonal surface $X$ that can be written using branch points only.  Then $\mathrm{Jac}^\ast=J_p^\ast$.
\end{theorem}
\begin{proof}
  Obviously  $J_p^\ast\subseteq\mathrm{Jac}^\ast$ by definition. Let $E\in \mathrm{Jac}^\ast$ then by the relations $p.a_i-p.a_j=0, a_1+\cdots+a_r=ra_i, \forall i=1,\ldots, r, a_i\in B(X)$ on  the $p$-gonal surface,  it follows that $\mathrm{Jac}^\ast=J_p^\ast$.
\end{proof}
%\section{Number of $m$-Spin Structures Supported on a Branch Set}
%\subsection{Hyperelliptic Case}
\begin{theorem}\label{J_m}
  Let $m$ be an even integer such that $m|2g-2$.  Out of the $m^{2g}$ $m$-spin structures on a hyperelliptic Riemann surface,  $2^{2g}$ of them can be written using branch points only.
\end{theorem}
\begin{proof}
  This follows immediately from Theorem~\ref{spin-Jm}, Theorem~\ref{J*-hyp}.
\end{proof}
%\subsection{$p$-Gonal case}

%Counting how many $p$-spin divisors of the above form reduces to looking at permissable values of $k$ and determining different integer partitions of $\sum c_i$.  This seemingly simple task, however, requires a lot of complicated and long calculations.  However, in the following we  calculate their number in a different way.

Let $X$ be a $p$-gonal surface of genus $g$ defined by the equation $y^p=(x-e_1)\cdots(x-e_r)$, where $x, y$ are complex variables, $r=np, n\in \mathbb{N}$ , $p$ is a prime, and $g=(p-1)(r-2)/2$.

\begin{theorem}
 Let $m$ be a multiple of $p$ such that $m|2g-2$.  The number of $m$-spin structures supported on the branch set of a $p$-gonal surface is the same as the number of $p$-spin structures  supported on the branch set.
\end{theorem}
\begin{proof}
    This follows immediately from Theorem~\ref{spin-Jm}, Theorem~\ref{J*-p}.
\end{proof}
One may ask how  many  $p$-spin structures on a $p$-gonal surface  supported on the branch set there are.  In the following we answer that by giving a presentation of  $J_p^\ast$.
\begin{theorem}\label{pres}
 The group $J_p^\ast$ has the following presentation $\displaystyle J_p^\ast=\langle a_i-a_r, i=1,\ldots, r-1| p.a_i-p.a_j=0, a_1+\cdots+a_r=ra_i, \forall i=1,\ldots, r, a_i\in B(X)\rangle .$
\end{theorem}
\begin{proof}
 % Define a map $\ph: G\longrightarrow J_p^\ast, g\longmapsto g $.  Obviously, this map is $1-1$.   Now we show this map is surjective i.e
 By definition $\langle a_i-a_r, i=1,\ldots, r-1| p.a_i-p.a_j=0, a_1+\cdots+a_r=ra_i, \forall i=1,\ldots, r\rangle \subseteq J_p^\ast$.
 We will show that  elements of $J_p^\ast$ can be written using the generators $ a_i-a_r$.  Let $E$ be  an element of $J_p^\ast$ i.e. $E=c_{i_1}a_{i_1}+\cdots+c_{i_k}a_{i_k}, \sum_{i=i_1}^{i_k} c_i=0, p.E=0$.  Then we have $c_{i_1}(a_{i_1}-a_r)+\cdots+c_{i_k}(a_{i_k}-a_r)=c_{i_1}a_{i_1}+\cdots+c_{i_k}a_{i_k}-(\sum_{i=i_1}^{i_k} c_i)a_r=E+0=E$.  Hence, $E\in \langle a_i-a_r, i=1,\ldots, r-1| p.a_i-p.a_j=0, a_1+\cdots+a_r=ra_i, \forall i=1,\ldots, r\rangle$.  Therefore, we have  $J_p^\ast=\langle a_i-a_r| p.a_i-p.a_j=0, a_1+\cdots+a_r=ra_i, \forall i=1,\ldots, r\rangle$.
\end{proof}
\begin{theorem}
  $J_p^\ast\cong \mathbb{Z}_p^{r-2}$.
\end{theorem}
\begin{proof}
  By the presentation given in Theorem~\ref{pres}, $J_p^\ast=\langle a_i-a_r,i=1,\ldots, r-1| p.a_i-p.a_j=0, a_1+\cdots+a_r=ra_i, \forall i=1,\ldots, r\rangle$.
Let $v_i=a_i-a_r, i=1,\ldots, r-1$.  Then $v_1+\cdots+v_{r-1}=a_1+\cdots+a_{r-1}-(r-1)a_r$.  Using the second relation in the presentation of $J_p^\ast$, it follows that $v_1+\cdots+v_{r-1}=0$. Therefore,  we have only $r-2$ independent generators and each is of order $p$ .  Hence,   $J_p^\ast\cong \mathbb{Z}_p^{r-2}$.
\end{proof}
\begin{corollary}
  Let $m$ be a multiple of $p$ such that $m|2g-2$.  The number of $m$-spin structures supported on the branch set of a $p$-gonal surface is $p^{r-2}$.
\end{corollary}

    \section{ Spin structures on Klein Surfaces}
\label{S:2}

In this section, we extend the definitions of spin structures on Riemann surfaces to Klein surfaces.  The main tool we use is complex doubles of Klein surfaces.  S. Natanzon has studied $m$-spin bundles on Riemann surfaces using Arf functions (functions on the space of homotopy classes of simple closed curves with values in $\mathbb{Z}_m$ ) and used that to describe moduli spaces.  His methods involves as well liftings of covering groups.  Later on, he generalized this to $m$-spin bundles on Klein surfaces, see~\cite{NP-HSK}, \cite{NP-MSH}, and \cite{NP-CSK}.  In this section, we give an equivalent definition  of $m$-spin structures on Klein surfaces using divisors.

\begin{definition}[Natanzon's definition of $m$-spin structures on Klein surfaces]
An $m$-spin structure on a Klein surface $Y=X/<\tau>$ is a pair $(L,\beta)$,
where $L$  is an $m$-spin (bundle) structure on the complex double $X$ (i.e. $ L^{\otimes m}=K$, where $K$ is the canonical line bundle on $X$ ) and $\beta$  is an anti-holomorphic
involution on  the line bundle $L$ such that the   diagram
  \begin{center}
\begin{tikzpicture}
    \node (L) at (0,0) {$L$};
    \node[right=of L] (L2) {$L$};
    \node[below=of L] (X) {$X$};
    \node[below=of L2] (X2) {$X$};
    \draw[->] (L)--(L2) node [midway,above] {$\beta$};
    \draw[->] (L2)--(X2) node [midway,right] {$\pi$};
    \draw[->] (X)--(X2) node [midway,above] {$\tau$};
    \draw[->] (L)--(X) node [midway,left] {$\pi$};
\end{tikzpicture}
\end{center}
commutes.
\end{definition}

Now we show this definition can be written using divisors.  Let $D=\sum a_i x_i$ be a divisor on a Riemann surface $X$ and  let $\tau$ be a symmetry on $X$. We define the action of $\tau$ on $D$ by $\tau(D)=\sum a_i \tau(x_i)$.
\begin{theorem}
A pair $(L,\beta)$ in the above sense is an $m$-spin structure on a Klein surface $Y=X/<\tau>$
\textbf{if and only if}

 $m.D_L\sim \bar{K}$ and  $D_L\sim \tau(D_L)$,  where $\bar{K}$ is a canonical divisor on $X$ and $D_L$ is the divisor associated with $L$.
\end{theorem}
\begin{proof}
  %Assume we have the commutative diagram.  It follows immediately that $2D_L=\bar{K}$.Let $(p,\xi) \in L$ where $p\in X,\xi\in \C$ , we have $\pi\circ \beta (p,\xi)=\tau\circ \pi (p, \xi)=\tau (p)$.Since $\beta(L)=L$, then $D_{\beta(L)}=D_L$.   For $p_i$ a point in $D_L$, $\pi\circ \beta(p_i,\xi_i)=\tau (p_i)\in D_{\beta(L)}$. Hence, $\tau(D_L)=D_{\beta(L)}=D_L$.
Assume we have the commutative diagram. Let $D=D_L$ denote the divisor, up to linear equivalence,  associated with  the $m$-spin bundle $L$ on the Riemann surface $X$. It follows immediately that $D$ is an $m$-spin divisor on $X$, $m.D_L=\bar{K}$.   We will show  that the $D$ is invariant under $\tau$, $\tau(D)= D$.  Let $\ph:X\longrightarrow L$ be a meromorphic section of $L$ for which $(\ph)=D$.  Consider the map $\psi=\beta\circ \ph \circ \tau$, $\psi:X\longrightarrow L$.  This map is meromorphic being the composition of two anti-holomorphic maps with a meromorphic map.  Furthermore, notice that
\begin{eqnarray*}
\pi\circ \psi &=&  (\pi\circ \beta)\circ \ph \circ \tau\\
              &=& \tau \circ\pi \circ \ph \circ \tau \text{\ (because\ } \tau\circ \beta=\beta\circ \tau \text{\ from the commutativity of the diagram)} \\
              &=& \tau\circ \mathrm{id}_X\circ \tau \text{\ (because\ } \ph \text{ is a meromorphic section of } X)\\
              &=& \tau^2=\mathrm{id}_X \text {\ (because }\tau \text {\ is an involution of } X),
\end{eqnarray*} hence $\psi$ is a meromorphic section of $L$.  Moreover, notice that if $D=\sum a_i x_i$, then $\tau(D)= \sum a_i \tau(x_i)$.    Notice that, by the commutativity of the diagram, $\psi=\beta\circ \ph \circ \tau$ and $\ph \circ \tau$ have corresponding zeros and poles  and since $D=(\ph)$, then $(\psi)=\tau(D)$. Since $\ph$ and $\psi$ are meromorphic sections of the same line bundle, they have linearly equivalent divisors, $D$ and $\tau(D)$.  Therefore, the divisor  $\tau(D)$ determines the same $m$-spin structure as the divisor $D$.

  On the other hand, assume that $D$ is an invariant $m$-spin divisor under $\tau$ i.e. $m.D=\bar{K}$, $D\tau(D)=D$.  Let $L_D$ be the line bundle associated with the $m$-spin divisor $D$. It follows that $ L^{\otimes m}=K$.  Furthermore, we define $\beta:L_D\longrightarrow L_D$ to be the map that satisfies  $\pi\circ \beta (p,\xi)=\tau (p), \forall p \in X$.  This forces the diagram  \begin{center}
\begin{tikzpicture}
    \node (L) at (0,0) {$L$};
    \node[right=of L] (L2) {$L$};
    \node[below=of L] (X) {$X$};
    \node[below=of L2] (X2) {$X$};
    \draw[->] (L)--(L2) node [midway,above] {$\beta$};
    \draw[->] (L2)--(X2) node [midway,right] {$\pi$};
    \draw[->] (X)--(X2) node [midway,above] {$\tau$};
    \draw[->] (L)--(X) node [midway,left] {$\pi$};
\end{tikzpicture}
\end{center}
to commute because $\pi\circ \beta (p,\xi)=\tau (p)=\tau\circ \pi (p, \xi)$.
\end{proof}
The authors of~\cite{pablo} proved, under different assumptions, related results  to the above theorem about non-orientable complex line bundles on Klein surfaces.
\begin{corollary}Let $Y=X/<\tau>$ be a Klein surface with  a complex double $X$.
 The set of $m$-spin structures on $Y$, denoted by $\Spin_m(Y)$,  are the $\tau$-invariant $m$-spin structures on $X$. In other words, $\Spin_m(Y)=\Spin_m^\tau (X)\doteq \{ D\in \Spin_m(X)| \tau(D)=D\}$.
\end{corollary}

\begin{example}\label{klein-example}
  Let $X$ be the hyperelliptic surface of genus $3$ defined by $y^2=z^7-z$.  The branch points are $-1,0,1,\infty, \xi_1=e^{\pi i/3}, \overline{\xi_1}, \xi_2=e^{2\pi i/3}, \overline{\xi_2}$. The canonical divisor is $K=4.\infty$. The Riemann surface $X$ has $64$ $2$-spin structures. Consider the involution $\tau:X\longrightarrow X, x\longmapsto \overline{x}$ on $X$.  The Klein surface $Y=X/<\tau>$ is non-orientable, see \cite{symmetries}.  The action of $\tau$ on the set of spin structures of $X$ leaves fixed, up to equivalence, the following divisors
$\theta_1=2.\infty,\  \theta_2=\infty+1,\  \theta_3=\infty+-1,\  \theta_4=\infty+0,\
  \theta_5=1+0,\  \theta_6=1+-1,\  \theta_7=0+-1,\ \theta_8=\xi_1+\overline{\xi_1},\
  \theta_9=\xi_2+\overline{\xi_2},\ \theta_{10}=-\infty+0+1+-1,\ \theta_{11}=-\infty+0+\xi_1+\overline{\xi_1},\ \theta_{12}=-\infty+0+\xi_2+\overline{\xi_2},
  \theta_{13}=-\infty+1+\xi_1+\overline{\xi_1},\  \theta_{14}=-\infty+1+\xi_2+\overline{\xi_2},\  \theta_{15}=-\infty+-1+\xi_1\overline{\xi_1},\  \theta_{16}=-\infty+-1+\xi_2+\overline{\xi_2}$.
Therefore, these $16$ divisors are the only $2$-spin structures on $Y$. This number of spin structures on $Y$ agrees with Natanzon's, see \cite{NP-HSK,NP-CSK}.
\end{example}

\section{$m$-Spin Structures on  Klein Surfaces with  Hyperelliptic Complex Doubles }\label{Klein-spin}
In Example~\ref{klein-example}, we determined and counted the number of $m$-spin divisors for a Klein surface whose complex double is hyperelliptic.  In the following, we  study the action of anti-holomorphic involutions on the $2$-torsion subgroup $J_2$ and determine the size of its invariant subgroup under certain involutions.
For a detailed classification of Klein Surfaces whose complex doubles is hyperelliptic, see~\cite{cirre}.  For a classification of automorphisms of hyperelliptic surfaces, see~\cite{symmetries}.   %cirre's article
\begin{theorem}\label{tau-spin}
  Let $X$ be a hyperelliptic surface with a defining polynomial whose coefficients are real and not all of its roots  are  real.  Assume $X$ has a symmetry $\tau$ induced by complex conjugation.  Then $\tau$ fixes $2^{g+k-1}$ \ $2$-spin structures on $X$, where $2k$ is the number of the real branch points.  In particular, the Klein surface $X/<\tau>$ has only $2^{g+k-1}$ \ $2$-spin structures.
\end{theorem}
\begin{proof}First, we study the action of $\tau$ on $J_2$.
By Theorem~\ref{J*-hyp} and Theorem~\ref{pres}, the group $J_2$ has the following presentation $J_2=\langle p_i-\infty,  i=1,\ldots, 2g+1| 2.p_i-2.\infty=0, p_1+\cdots+p_{2g+1}=(2g+1).\infty, \forall i=1,\ldots, 2g+1\rangle .$
Let $a_i, i=1, \ldots, 2k$ denote the set of  the $2k$ real branch points and by convention the point at $\infty$ is real,$a_1=\infty$. Since the defining polynomial has real coefficients, the number of non-real branch points is, $2g+2-2k$, even. Let $b_j, j=1, 2g+2-2k$ denote the non-real branch points.  Moreover, if $b_j$ is a root, then so is its conjugates, $\overline{b_j}$.  Therefore, the non-real branch points are $b_j,\overline{b_j}, j=1, \ldots,g+1-k$.

    %Let $B$ denote the branch set and let $E_g$ be the quotient $\{T|T\subset B, |T|\text{\ is even}\}/\sim$ where $T_1\sim T_2 \Longleftrightarrow T_1=T_2 \text{\ or } T_1=T^c_2 $.
  %Following \cite{invariant} and \cite{dolg}, we define the map $\alpha:E_g\longmapsto J_2, T\longmapsto \sum_{a_i\in T} a_i-\infty$, and this map is an isomorphism.Hence, from the correspondence in Theorem~\ref{spin-Jm}, counting the number of invariant spin structures under $\tau$ reduces to counting how many elements in $E_g$ are fixed by $\tau$.
Note that if $E\in J_2$ consists of only real points, then $\tau(E)=E$.  Moreover, if $E$ contains points that are not real, provided that if a point appears then so does its conjugate, then $\tau(E)=E$.  Therefore, the group of invariant $J_2$ elements under $\tau$, denoted by $J_2^\tau$, is
given by
$J_2^\tau=\langle v_i, w_j, i=2,\ldots, 2k, j=1,\ldots, g+1-k| 2.v_i=0, 2.w_j=0, v_2+\cdots+v_{2k}+w_1+\cdots+w_{g+1-k}=0, \forall i=2,\ldots, 2k, j=1,\ldots, g+1-k\rangle $
where $v_i=a_i-\infty, i=2,\ldots,2k$ and $w_i=b_i+\overline{b_i}-2\infty, i=i=2k+1, g+1$.  The second relation allows us to get rid of one generator.  Hence, $J_2^\tau$ has only $g+k-1$ independent generators each of order $2$.  Therefore,  $J_2^\tau \cong \mathbb{Z}_2^{g+k-1}$.

The proof remains the same if  $\infty$ was not a branch point.
\end{proof}
Our number agrees with that of Natanzon's in  \cite{NP-HSK}.
\begin{corollary}
Let $X$ be a hyperelliptic surface with a defining polynomial whose coefficients are real and some of its roots  are not real.  Assume $X$ has a symmetry $\tau$ induced by complex conjugation.
  Let $m$ be an even integer such that $m|2g-2$.  Out of the $m^{2g}$ $m$-spin structures on $X$,  $2^{2g}$ of them can be written using branch points only from which $\tau$ fixes $2^{g+k-1}$, where $2k$ is the number of the real branch points.  In particular, the Klein surface $X/<\tau>$ has only $2^{g+k-1}$\  $m$-spin structures in terms of branch points.
\end{corollary}
\begin{proof}
  This follows immediately from the above theorem and from Theorem~\ref{J_m}.
\end{proof}
\begin{theorem}\label{tau-spin-2}
  Let $X$ be a hyperelliptic surface of an odd genus, $g$.  Let $\tau$ be an involution on $X$ induced by the antipodal map $-1/\overline{z}$.  Then the number of invariant $2$-spin structures of $X$ under $\tau$ is $2^g$.   In particular, the Klein surface $X/<\tau>$ has only $2^{g}$ \ $2$-spin structures.
\end{theorem}
\begin{proof}
The antipodal map  fixes no branch points.  We study the action of $\tau$ on $J_2$.  By Theorem~\ref{J*-hyp} and Theorem~\ref{pres}, the group $J_2$ has the following presentation $J_2=\langle p_i-\infty,  i=1,\ldots, 2g+1| 2.p_i-2.\infty=0, p_1+\cdots+p_{2g+1}=(2g+1).\infty, \forall i=1,\ldots, 2g+1\rangle .$

The antipodal involution only fixes elements of $J_2$ for which both a point and its antipodal appear.  Therefore, $\tau$ fixes elements of the type $a_i+\tau(a_i)-2.\infty$. Hence, generators for  the group of invariant $J_2$ elements under $\tau$, denoted by $J_2^\tau$, are
given by $a_i+\tau(a_i)-2.\infty$.  In other words, a presentation for $J_2^\tau$ is given by $J_2^\tau=<v_i, i=1,\cdots, g|2.v_i=0>$.  Hence, $J_2^\tau\cong \mathbb{Z}_2^{g}$.  If  $\infty$ was not a branch point,  but rather we had a branch point $p_{2g+2}$, then the  proof remains the same.

\end{proof}

\begin{corollary}
 Let $X$ be a hyperelliptic surface of an odd genus, $g$.  Let $\tau$ be an involution on $X$ induced by the antipodal map $-1/\overline{z}$.
  Let $m$ be an even integer such that $m|2g-2$.  Out of the $m^{2g}$ $m$-spin structures on $X$,  $2^{2g}$ of them can be written using branch points only from which $\tau$ fixes $2^{g}$.  In particular, the Klein surface $X/<\tau>$ has only $2^{g}$\  $m$-spin structures in terms of branch points.
\end{corollary}
\begin{proof}
  This follows immediately from the above theorem and from Theorem~\ref{J_m}.
\end{proof}
\section{Invariant $m$-Spin Structures on a Hyperelliptic Surface under an Automorphism  }\label{invariant}
  According to~\cite{symmetries}, since a hyperelliptic surface $X$ is a branched cover of the Riemann sphere, automorphisms of a hyperelliptic surface  are  closely related to M\"obius transformations of the Riemann sphere.
It turns out that the group generated by the hyperelliptic involution of a hyperelliptic surface $I_2:X\longrightarrow X$ is a central normal subgroup of  $\mathrm{Aut}(X)$ and $<I_2>\cong \mathbb{Z}_2$.  Furthermore, the group $\mathrm{Aut}(X)/<I_2>$ is a M\"obius group and $\mathrm{Aut}(X)/<I_2>=\{f\in \mathrm{Aut}(\hC)| f(B)=B\}$, where $B$ denotes the branch set of the hyperelliptic surface.  Therefore, when we speak of such automorphisms, we will think of them as (lifts of) M\"obius transformations of the sphere that leaves fixed the branch set.
%\subsection{Action of Odd Order Automorphisms on $J_2$}

In this section, we will study invariant $m$-spin divisors under an odd order automorphism of a hyperelliptic surface.  Similar calculations can be done for even order automorphisms.  Similarly, we can extend these calculations to $p$-gonal surfaces.
\begin{theorem}
\label{none-fixed}
  Let $f$ be an automorphism of a hyperelliptic Riemann surface of genus $g$.  Assume $f$ has an odd order $n$.  Assume $f$ fixes no branch points.
  Then the  group of invariant elements of  $J_2$ under $f$, denoted by $J_2^f$, is isomorphic to $\mathbb{Z}_2^{\frac{2g+2}{n}-2}$.  In other words, $f$ fixes $2^{\frac{2g+2}{n}-2}$\ $2$-spin structures.
\end{theorem}
This theorem was proved using different methods in \cite{invariant}.  We prove it here using similar methods to the ones we used to prove Theorems~\ref{tau-spin} and ~\ref{tau-spin-2}.
\begin{proof}
Recall that  the group $J_2$ has the following presentation
$J_2=\langle p_i-p_{2g+2},  i=1,\ldots, 2g+2| 2.p_i-2.p_{2g+2}=0, p_1+\cdots+p_{2g+2}=(2g+2).p_{2g+2}, \forall i=1,\ldots, 2g+2\rangle.$
  Elements of $J_2$ that are fixed by $f$ must contain a point and its orbit,  $p_i+f(p_i)+\cdots+f^{n-1}(p_i)-(p_{j}+f(p_{j})+\cdots+f^{n-1}(p_{j}))$.  Since $f$ doesn't fix any branch points, we have $\frac{2g+2}{n}$ orbits.  Without loss of generality, we may assume that the branch points are ordered so that $p_1,\ldots, p_{\frac{2g+2}{n}-1},  p_{2g+2}$ are representatives of the $\frac{2g+2}{n}$ orbits.  Let $v_i=p_i+f(p_i)+\cdots+f^{n-1}(p_i)-(p_{2g+2}+f(p_{2g+2})+\cdots+f^{n-1}(p_{2g+2})), i=1,\ldots, \frac{2g+2}{n}-1,  2g+2$.  Then $ v_1+\cdots+ v_{\frac{2g+2}{n}-1}+v_{2g+2}=p_1+\cdots+p_{2g+2}-\frac{2g+2}{n}.((p_{2g+2}+f(p_{2g+2})+\cdots+f^{n-1}(p_{2g+2})))$.  Since $n$ is odd, $n$ must be a factor of $g+1$. Therefore, we could use the relation $2.p_i-2.p_{2g+2}$  to get $\frac{2g+2}{n}.((p_{2g+2}+f(p_{2g+2})+\cdots+f^{n-1}(p_{2g+2})))=\frac{g+1}{n}(2.p_{2g+2}+2.f(p_{2g+2})+\cdots+2.f^{n-1}(p_{2g+2}))=
  \frac{g+1}{n}(2n(p_{2g+2}))=(2g+2).p_{2g+2}$.  Hence  $v_1+\cdots+ v_{\frac{2g+2}{n}-1}+v_{2g+2}=p_1+\cdots+p_{2g+2}-(2g+2).p_{2g+2}$.
   Notice that $v_{2g+2}=0$ by definition. The second relation in the  presentation of $J_2$ implies that $ v_1+\cdots+ v_{\frac{2g+2}{n}-1}=0$, hence we have only $\frac{2g+2}{n}-2$ independent elements of the form  $v_i=p_i+f(p_i)+\cdots+f^{n-1}(p_i)-(p_{2g+2}+f(p_{2g+2})+\cdots+f^{n-1}(p_{2g+2}))$.
  Therefore, the  group of invariant elements of  $J_2$ under $f$, denoted by $J_2^f$, has the following presentation $J_2^f=<v_i, i=1, \cdots, \frac{2g+2}{n}-2|2v_i=0>\ \cong \mathbb{Z}_2^{\frac{2g+2}{n}-2}$.

\end{proof}

\begin{corollary}
   Let $f$ be an automorphism of a hyperelliptic Riemann surface of genus $g$.  Assume $f$ has an odd order $n$.  Assume $f$ fixes no branch points.
  Then the  group of invariant elements of  $J^\ast_m$ under $f$, denoted by $J_m^{\ast\ f}$, is isomorphic to $\mathbb{Z}_2^{\frac{2g+2}{n}-2}$.  In other words $f$ fixes $2^{\frac{2g+2}{n}-2}$\ $m$-spin  structures supported on the branch set.
\end{corollary}
\begin{proof}
    This follows immediately from the above theorem and from Theorem~\ref{J_m}.
\end{proof}
\begin{theorem}\label{one-fixed}
  Let $f$ be an automorphism of a hyperelliptic Riemann surface of genus $g$.  Assume $f$ has an odd order $n$.  Assume $f$ fixes only one branch point.
  Then the  group of invariant elements of  $J_2$ under $f$, denoted by $J_2^f$, is isomorphic to $\mathbb{Z}_2^{\frac{2g+1}{n}-1}$.  In other words $f$ fixes $2^{\frac{2g+2}{n}-1}$\ $2$-spin structures.
\end{theorem}
This theorem was also proved using different methods in \cite{invariant}.  Yet, We prove it here using our methods as in the proof of Theorem~\ref{none-fixed}.
\begin{proof}
Recall that  the group $J_2$ has the following presentation
$J_2=\langle p_i-p_{2g+2},  i=1,\ldots, 2g+2| 2.p_i-2.p_{2g+2}=0, p_1+\cdots+p_{2g+2}=(2g+2).p_{2g+2}, \forall i=1,\ldots, 2g+2\rangle.$
Let $p_{2g+2}$ denote the fixed point of $f$.
  Elements of $J_2$ that are fixed by $f$ are of the form  $p_i+f(p_i)+\cdots+f^{n-1}(p_i)-(p_{j}+f(p_{j})+\cdots+f^{n-1}(p_{j}))$ or $p_i+f(p_i)+\cdots+f^{n-1}(p_i)-n.p_{2g+2}$. Elements of the first type can be written in terms of the second type so it is enough to consider elements of the second type.  We have $\frac{2g+1}{n}+1$ orbits.  Without loss of generality, we may assume that the branch points are ordered so that $p_1,\ldots, p_{\frac{2g+1}{n}},  p_{2g+2}$ are representatives of the $\frac{2g+1}{n}+1$ orbits plus the fixed point having an orbit on its own.  Let $v_i=p_i+f(p_i)+\cdots+f^{n-1}(p_i)-(p_{2g+2}+f(p_{2g+2})+\cdots+f^{n-1}(p_{2g+2})), i=1,\ldots, \frac{2g+1}{n}$ and $v_{2g+2}=p_{2g+2}-p_{2g+2}=0$.  Then $ v_1+\cdots+ v_{\frac{2g+1}{n}}+v_{2g+2}=p_1+\cdots+p_{2g+2}-(\frac{2g+1}{n}n.p_{2g+2}-p_{2g+2})=p_1+\cdots+p_{2g+2}-(2g+2).p_{2g+2}$.  Using the second relation in the presentation of $J_2$, we have   $v_1+\cdots+ v_{\frac{2g+1}{n}}+v_{2g+2}=0$.
   Notice that $v_{2g+2}=0$ by definition. The second relation in the  presentation of $J_2$ implies that $ v_1+\cdots+ v_{\frac{2g+1}{n}}=0$, hence we have only $\frac{2g+1}{n}-1$ independent elements of the form $v_i=p_i+f(p_i)+\cdots+f^{n-1}(p_i)-(p_{2g+2}+f(p_{2g+2})+\cdots+f^{n-1}(p_{2g+2}))$.
  Therefore, the  group of invariant elements of  $J_2$ under $f$, denoted by $J_2^f$, has the following presentation $J_2^f=<v_i, i=1, \cdots, \frac{2g+1}{n}-1|2v_i=0>\ \cong \mathbb{Z}_2^{\frac{2g+1}{n}-1}$.

\end{proof}
\begin{corollary}
 Let $f$ be an automorphism of a hyperelliptic Riemann surface of genus $g$.  Assume $f$ has an odd order $n$.  Assume $f$ fixes only one branch point.  Then the  group of invariant elements of  $J^\ast_m$ under $f$, denoted by $J_m^{\ast\ f}$, is isomorphic to $\mathbb{Z}_2^{\frac{2g+1}{n}-1}$.  In other words $f$ fixes $2^{\frac{2g+1}{n}-1}$\ $m$-spin  structures supported on the branch set.
\end{corollary}
\begin{proof}
    This follows immediately from the above theorem and from Theorem~\ref{J_m}.
\end{proof}

\begin{theorem}\label{two-fixed}
  Let $f$ be an automorphism of a hyperelliptic Riemann surface of genus $g$.  Assume $f$ has an odd order $n$.  Assume $f$ fixes two branch points.
  Then the  group of invariant elements of  $J_2$ under $f$, denoted by $J_2^f$, is isomorphic to $\mathbb{Z}_2^{\frac{2g}{n}}$.  In other words $f$ fixes $2^{\frac{2g}{n}}$\ $2$-spin structures.
\end{theorem}
This theorem was also proved using different methods in \cite{invariant}.   We prove it here using our methods as in the proof of Theorems~\ref{none-fixed} and~\ref{one-fixed}.
\begin{proof}
Recall that  the group $J_2$ has the  presentation
$J_2=\langle p_i-p_{2g+2},  i=1,\ldots, 2g+2| 2.p_i-2.p_{2g+2}=0, p_1+\cdots+p_{2g+2}=(2g+2).p_{2g+2}, \forall i=1,\ldots, 2g+2\rangle.$
Let $p_{2g+1},p_{2g+2}$ denote the fixed point of $f$.
  Non-trivial Elements of $J_2$ that are fixed by $f$ are those of the forms
  \begin{enumerate}
    \item $p_i+f(p_i)+\cdots+f^{n-1}(p_i)-(p_{j}+f(p_{j})+\cdots+f^{n-1}(p_{j})), i\neq 2g+1, 2g+2$,
    \item $p_i+f(p_i)+\cdots+f^{n-1}(p_i)-n.p_{2g+1}, i\neq 2g+1, 2g+2$,
    \item $p_i+f(p_i)+\cdots+f^{n-1}(p_i)-n.p_{2g+2}, i\neq 2g+1, 2g+2$,
    \item $p_{i}-p_j, \{i,j\}=\{2g+1, 2g+2\}$.
  \end{enumerate}
 We observe that elements of the first two types can be written in terms of these of the third and fourth types so it is enough to consider elements of the last two types.  We have $\frac{2g}{n}$ orbits of the non-fixed points.  Without loss of generality, we may assume that the branch points are ordered so that $p_1,\ldots, p_{\frac{2g}{n}}$ are representatives of the $\frac{2g}{n}$ orbits.  Let $v_i=p_i+f(p_i)+\cdots+f^{n-1}(p_i)-n.p_{2g+2}, i=1,\ldots, \frac{2g}{n}$ and $v_{2g+1}=p_{2g+1}-p_{2g+2}, v_{2g+2}=p_{2g+2}-p_{2g+2}$.  Notice that $v_{2g+2}=0$.  Furthermore, $ v_1+\cdots+ v_{\frac{2g}{n}}+v_{2g+1}+v_{2g+2}=p_1+\cdots+p_{2g+2}-\frac{2g}{n}n.p_{2g+2}-2p_{2g+2}=p_1+\cdots+p_{2g+2}-(2g+2).p_{2g+2}$.  Using the second relation in the presentation of $J_2$, we have   $v_1+\cdots+ v_{\frac{2g}{n}}+v_{2g+1}+v_{2g+2}=0$.
   Notice that $v_{2g+2}=0$ by definition. The second relation in the  presentation of $J_2$ implies that $ v_1+\cdots+ v_{\frac{2g}{n}}+v_{2g+1}=0$, hence we have only $\frac{2g}{n}$ independent elements of the form $v_i=p_i+f(p_i)+\cdots+f^{n-1}(p_i)-n.p_{2g+2}$.
  Therefore, the  group of invariant elements of  $J_2$ under $f$, denoted by $J_2^f$, has the following presentation $J_2^f=<v_i, i=1, \cdots, \frac{2g}{n}|2v_i=0>\ \cong \mathbb{Z}_2^{\frac{2g}{n}}$.

\end{proof}
%\subsection{Even Order Automorphism}

\begin{theorem}
  Let $I_p$ be the $p$-involution on a $p$-gonal surface $X$.  Then $I_p$ fixes every spin structure supported on the branch set.
\end{theorem}
\begin{proof}
  This follows immediately from the fact that  the $p$-involution $I_p$ fixes branch points.
\end{proof}
 One may ask if $I_p$ is the only automorphism with such property.  We will investigate invariant spin structure on $p$-gonal surfaces under automorphisms in a future work.
\section{Mumford's Formula for Spin Structures on Hyperelliptic Surfaces}\label{mumf}
%\subsection{hyperelliptic}
In this section, we review a formula obtained by Mumford in \cite{mumford}.  This formula gives a unique representation of $2$-spin divisors (theta characteristics) on hyperelliptic surfaces.   Moreover, we derive a similar presentation for $m$-spin divisors, when $m$ is even.

\subsection*{Observations}
Let $X$ be a hyperelliptic surface.  We have seen that the canonical divisor is given by  $K_X \sim \pi^*(K_\infty)+R_\pi\sim -2D+p_1+\cdots+ p_{2g+2}\sim -2D+(g+1)D=(g-1)D$, where $D=2.p_i, p_i\in B(X)$.  Recall we have the relations $2p_i=2p_j, \forall p_i, p_j \in B(X)$ and $p_{1} +\cdots+p_{2g+2} = (2g+2).p_j$.
%\begin{theorem}

Following \cite{dolg,invariant}, we make the following observations. For a subset $T$ of the branch set, $T\subset B(X)$, define the divisor  $\alpha(T)=(\sum_{p_i\in T}p_i)-\mathrm{card}(T)$.  It follows immediately that $2\alpha_T = 2(\sum_{p_i\in T}p_i)-2\mathrm{card}(T)=2\mathrm{card}(T)-2\mathrm{card}(T)=0$.  Hence,  we have  $\alpha(T)\in J_2$.  Furthermore,  as we pointed out in Section~\ref{hyp-p}, the divisor of the meromorphic function  $h(x,y)=y/(x-e_j)^{g+1}$ is $\divv (h)=p_{1} +\cdots+p_{2g+2} - (2g+2).p_j=0$. Therefore,  $\alpha(B(X))=(\sum_{p_i\in B(X)}p_i)-\mathrm{card}(B(X))=p_{1} +\cdots+p_{2g+2} - (2g+2).p_j=\divv(h)=0$.  Hence, we can assume $\mathrm{card}(T)$ is even, for any subset $T\subset B(X)$, because we can always add the zero divisor $p_j-p_j$ to $\alpha(T)$.   We  also notice that the relations $2p_i=2p_j$ and $p_{1} +\cdots+p_{2g+2} = (2g+2).p_j$ imply that $p_{i_1} +\cdots+p_{i_{g+1}} = p_{j_1} +\cdots+p_{j_{g+1}}$, therefore $\alpha(T)=\alpha(T^c)$, where $T^c$ is the complement of $T$ in $B(X)$.

Let $E_g=\{T|T\subset B(X), \mathrm{card}(T)\text{\ is even}\}/\approx$, where $T_1\approx T_2$ if $T_1=T_2$ or $T_1=T_2^c$.  This space can be given an additive group structure by defining the sum of two elements $T_i, T_j$ to be $T_i+T_j=(T_i\cup T_j)\setminus (T_i\cap T_j)$.  By the above observations, we see that the map $E_g\longrightarrow J_2,  T\longmapsto \alpha(T)$ is an isomorphism.

\subsection*{Mumford Formula for $2$-Spin Divisors}
  \begin{theorem}[Mumford's Formula]\label{mumf-2}
   Every $2$-spin divisor (theta characteristic) is of the form
$\theta = kD + p_{i_1} + \cdots + p_{i_{g-1-2k}}$
for $-1 \le k \le g-1$.  Moreover this representation
is unique if $k \ge 0$ and subject to a single relation when $k = -1$,  $-D+p_{i_1} +\cdots+p_{i_{g+1}} = -D+p_{j_1} +\cdots+p_{j_{g+1}}$.
\end{theorem}
\begin{proof}
  In this proof we follow \cite{invariant}. First, we notice that the relations $2p_i=2p_j$ and $p_{1} +\cdots+p_{2g+2} = (2g+2).p_j$ imply that $p_{i_1} +\cdots+p_{i_{g+1}} = p_{j_1} +\cdots+p_{j_{g+1}}$.  Moreover, any divisor of the above form is a $2$-spin divisor because $2\theta = 2kD +2 p_{i_1} + \cdots + 2p_{i_{g-1-2k}}=(2k+g-1+2k)D=(g-1)D=K$. Notice that $g-1-2k=g+1$ when $k=-1$. In that case, the representation is not unique because $p_{i_1} +\cdots+p_{i_{g+1}} = p_{j_1} +\cdots+p_{j_{g+1}}$  for distinct  $i$'s and $j$'s  as pointed out in the observations above. Therefore, when $k=-1$, we have $\frac{1}{2}\binom{2g+2}{g+1}$ different $2$-spin divisors.

  Moreover, the representation is unique when $0\le k \le \frac{g-1}{2}$ because the presentation will contain strictly less than $g+1$ points.  In this case, we have $\sum_{k=0}^{\frac{g-1}{2}}\binom{2g+2}{g-1-2k}$.  Therefore,  Mumford's formula  gives $\frac{1}{2}\binom{2g+2}{g+1}+\sum_{k=0}^{\frac{g-1}{2}}\binom{2g+2}{g-1-2k}=2^{2g}$ different $2$-spin divisors.
\end{proof}
\subsection*{ A Formula for   $m$-Spin Divisors Supported on the Branch set}

We generalize Mumford's formula to the case of $m$-spin divisors supported on the branch set.   Let $\theta=kD+p_{i_{1}}+\cdots + p_{i_{\ell}}$. We will find  values of $\ell$  for which $\theta$ is an $m$-spin structure.  Multiplying by $m$, we get $m\theta=2n\theta=2nkD+n(2.p_{i_{1}}+\cdots +2. p_{i_{\ell}})=(2nk+\ell n)D$.  For $\theta$ to be an $m$-spin structure, we need $(2nk+\ell n)D =(g-1)D$. Solving for $\ell$, we get $\ell=\frac{g-1-2nk}{n}$.
Hence,
  $  m\theta=2n\theta=2nkD+(g-1-2nk)D=(g-1)D=K_X  $ and therefore $\theta$ is an $m$-spin divisor.
\begin{theorem}[General Formula for $m$-Spin Divisors Supported on the Branch Set]\label{mumf-m}
Assume $m=2n$ and $n\Big| g-1$.  Every $m$-spin divisor supported on the branch set is of the form $\displaystyle\theta= kD+p_{i_{1}}+\cdots + p_{i_{\frac{g-1-2nk}{n}}}$, where $\frac{g(n-1)+n+1}{-2n}\le k \le \frac{g-1}{2n}$.  This representation is unique when $\frac{g(n-1)+n+1}{-2n}< k \le \frac{g-1}{2n}$ and subject to the relation $\frac{g(n-1)+n+1}{-2n}D+p_{i_1} +\cdots+p_{i_{g+1}} = \frac{g(n-1)+n+1}{-2n}D+p_{j_1} +\cdots+p_{j_{g+1}}$ when $k=\frac{g(n-1)+n+1}{-2n}$.
\end{theorem}
\begin{proof}
 It follows immediately that $\theta$ is an $m$-spin divisor because $  m\theta=2n\theta=2nkD+(g-1-2nk)D=(g-1)D=K_X$.   Furthermore, the representation contains $g+1$ points  when $k=\frac{g(n-1)+n+1}{-2n}$.  In this case, the representation is not unique because $p_{i_1} +\cdots+p_{i_{g+1}} = p_{j_1} +\cdots+p_{j_{g+1}}$ for distinct  $i$'s and $j$'s  as pointed out in the observations above.  Therefore, when $k=\frac{g(n-1)+n+1}{-2n}$, we have $\frac{1}{2}\binom{2g+2}{g+1}$ different $m$-spin divisors supported on the branch set.

 Moreover, the representation is unique when $\frac{g(n-1)+n+1}{-2n}< k \le \frac{g-1}{2n}$ because the presentation will contain strictly less than $g+1$ points.  In this case, we have  $$\displaystyle \sum_{k>\frac{g(n-1)+n+1}{-2n}}^{\frac{g-1}{2n}}\binom{2g+2}{\frac{g-1-2nk}{n}}$$ different $m$-spin divisors supported on the branch set.  Therefore, the representation in the theorem determines $$\frac{1}{2}\binom{2g+2}{g+1}+\sum_{k>\frac{g(n-1)+n+1}{-2n}}^{\frac{g-1}{2n}}\binom{2g+2}{\frac{g-1-2nk}{n}}=2^{2g}$$ different $m$-spin divisors supported on the branch set.

 \end{proof}
 \section{Future Work}
 When viewing $2$-spin structures as divisors on compact Riemann surfaces, there is a nice relation between spin structures and certain classes of meromorphic differentials.  In particular, the authors of \cite{spinrep} have proved the following theorem.
 \begin{theorem}
   Let $X$ be a compact Riemann surface.  There is a natural bijection between $2$-spin structures on $X$ and classes of non-zero meromorphic differentials with even zeros and poles where $\omega_1\sim \omega_2$ when $\omega_1/\omega_2=f^2$, where $f$ is a meromorphic function on $X$.
 \end{theorem}
We have proved a generalization of the above theorem to the case of $m$-spin structures.  We are working on applications of that to the case of hyperelliptic surfaces. This shall appear in our next paper.
This work is related to
\cite{almalkithesis}.

\bibliographystyle{plain}
 \bibliography{mybib}

\end{document}